\documentclass[12pt]{article}
\usepackage{fullpage}
\usepackage{amssymb}
\usepackage{amsmath}
\usepackage{amsthm}
\usepackage{stmaryrd}
\usepackage{xypic}
\usepackage{setspace}

\newtheorem{theorem}{Theorem}[section]

\newtheorem{proposition}[theorem]{Proposition}
\newtheorem{corollary}[theorem]{Corollary}

\newcommand{\tref}[1]{Theorem~\textup{\ref{thm:#1}}}
\newcommand{\pref}[1]{Proposition~\textup{\ref{prop:#1}}}
\newcommand{\cref}[1]{Corollary~\textup{\ref{cor:#1}}}
\newcommand{\eref}[1]{Equation~\textup{\ref{eq:#1}}}

\newcommand{\tx}[1]{\textrm{#1}}
\newcommand{\lcm}{\textrm{lcm}}
\newcommand{\gp}[2]{\langle #1 \mid #2 \rangle}
\newcommand{\s}{\sigma}
\newcommand{\G}{\Gamma}
\newcommand{\calK}{\mathcal K}
\newcommand{\calL}{\mathcal L}
\newcommand{\calP}{\mathcal P}
\newcommand{\calPd}{\mathcal P^{\delta}}
\newcommand{\calQ}{\mathcal Q}

\newcommand{\GP}{\Gamma^+(\calP)}
\newcommand{\GPd}{\Gamma^+(\calP^{\delta})}
\newcommand{\ch}[1]{\overline{#1}}
\newcommand{\GcP}{\Gamma^+(\ch{\calP})}

\newcommand{\GQ}{\Gamma^+(\calQ)}

\newcommand{\GK}{\Gamma^+(\calK)}
\newcommand{\GPQ}{\G^+(\calP \mix \calQ)}
\newcommand{\GcPQ}{\G^+(\calP) \comix \G^+(\calQ)}
\newcommand{\mix}{\diamond}
\newcommand{\comix}{\boxempty}
\newcommand{\comment}[1]{}
\newcommand{\eps}{\epsilon}
    
\begin{document}

\title{Constructing Self-Dual Chiral Polytopes}

\author{Gabe Cunningham\\
Department of Mathematics\\
Northeastern University\\
Boston, Massachusetts,  USA, 02115
}

\date{ \today }
\maketitle

\begin{abstract}
An abstract polytope is \emph{chiral} if its automorphism group has two orbits on the flags, such that adjacent 
flags belong to distinct orbits. There are still few examples of chiral polytopes, and few constructions
that can create chiral polytopes with specified properties. In this paper, we show how to build self-dual 
chiral polytopes using the mixing construction for polytopes.



\end{abstract}

\section{Introduction}
	The study of abstract polytopes is a growing field, uniting combinatorics with geometry and
	group theory. One particularly active area of research is the study of \emph{chiral polytopes}.
	Chiral polytopes are ``half-regular'': the action of the automorphism group on the flags has two orbits, and 
	adjacent flags belong to distinct orbits. Choosing one of the two orbits amounts to choosing an ``orientation'',
	and we say that the two orientations are \emph{enantiomorphic} or that they are \emph{mirror images}
	(of each other). 
	
	Chiral maps (also called irreflexible maps) have been studied for some time (see \cite{cm}), and the study of 
	chiral maps and hypermaps continues to yield interesting developments (for example, see \cite{aj}).
	However, it was only with the introduction of abstract polytopes that the notion of chirality was defined for 
	structures of ranks 4 and higher \cite{chiral}.
	
	The well-known geometric operation of dualizing a convex polytope (resulting in the polar polytope) has a
	simple analogue with abstract polytopes. In fact, the dual of an abstract polytope is obtained simply
	by reversing the partial order of the faces. When a polytope is isomorphic to its dual, it is said
	to be \emph{self-dual}. A self-dual chiral polytope is \emph{properly self-dual} if its dual has the same orientation
	(i.e., choice of flag orbit), and \emph{improperly self-dual} otherwise. Self-duality of chiral polytopes
	is studied in detail in \cite{chiral-dual}.
	
	There are still few known concrete examples of finite chiral polytopes. Only recently have we found
	general ways of building new chiral polytopes of higher ranks (see \cite{const}, \cite{chp}, and \cite{pel}). 
	In this paper, we use the mixing construction, introduced in \cite{mix-face} for regular polytopes and then in 
	\cite{const} for chiral 
	polytopes. To construct self-dual chiral polytopes, we mix a chiral polytope with its dual or with
	the mirror image of its dual. This always yields something which is self-dual, but it may not
	be chiral or polytopal. Our goal, then, is to find simple criteria that guarantee that we do actually get
	a self-dual chiral polytope.
	
	For our construction, confirming polytopality seems somewhat more difficult than confirming chirality. Though 
	there are some
	general results on when the mix of two polytopes is polytopal, they tend not to work well when applied
	to a polytope and its dual. They are particularly ill-suited to polytopes of even rank. In fact,
	in some cases we are able to prove that the mix of a polytope of even rank with its dual is not polytopal.
	On the other hand, our construction works particularly well with chiral polyhedra,
	because the mix of two chiral polyhedra is always a polyhedron (i.e., polytopal).
	
	We start by giving some background information on duality of abstract regular and chiral polytopes in 
	Section~2.  In Section~3, we introduce the mixing operation for chiral and directly regular polytopes,
	and we give a few results for when the mix of two polytopes is again a polytope.
	Then, in Section~4, we examine the mix of a polytope with its dual, proving that the mix is self-dual and
	determining some criteria for when the mix is polytopal. In Section~5, we determine several simple criteria
	for when the mix of a polytope with its dual is chiral. Finally, we end Section~5 by showing broad
	circumstances under which the mix of a chiral polytope with its dual is a chiral self-dual polytope, and
	we give some concrete examples in Section~6.

\section{Polytopes}

	General background information on abstract polytopes can be found in \cite[Chs. 2, 3]{arp}, and information
	on chiral polytopes specifically can be found in \cite{chiral}.
	Here we review the concepts essential for this paper.

	\subsection{Definition of a polytope}

		Let $\calP$ be a ranked partially ordered set whose elements will be called \emph{faces}. 
		The faces of $\calP$ will range in rank from $-1$ to $n$, and a face of rank $j$ is called a 
		\emph{$j$-face}. The $0$-faces, $1$-faces, and $(n-1)$-faces are also 
		called \emph{vertices}, \emph{edges}, and \emph{facets}, respectively. A \emph{flag} of
		$\calP$ is a maximal chain. We say that two flags are \emph{adjacent} if they differ 
		in exactly one face, and that they are \emph{$j$-adjacent} if they differ only in their $j$-face. 
		If $F$ and $G$ are faces of $\calP$
		such that $F \leq G$, then the \emph{section} $G / F$ consists of those faces $H$ such that
		$F \leq H \leq G$.
		
		We say that $\calP$ is an \emph{(abstract) polytope of rank $n$}, also called an \emph{$n$-polytope}, 
		if it satisfies the following four properties:
		
		\begin{enumerate}
		\item There is a unique greatest face $F_n$ of rank $n$ and a unique least face $F_{-1}$ of rank $-1$.
		\item Each flag of $\calP$ has $n+2$ faces.
		\item $\calP$ is \emph{strongly flag-connected}, meaning that if $\Phi$ and $\Psi$
		are two flags of $\calP$, then there is a sequence of flags $\Phi = \Phi_0, \Phi_1, \ldots, \Phi_k = \Psi$ 
		such that for $i = 0, \ldots, k-1$, the flags $\Phi_i$ and $\Phi_{i+1}$ are adjacent, and each
		$\Phi_i$ contains $\Phi \cap \Psi$.
		\item (Diamond condition): Whenever $F < G$, where $F$ is a $(j-1)$-face and $G$ is a $(j+1)$-face
		for some $j$, then there are exactly two $j$-faces $H$ with $F < H < G$.
		\end{enumerate}
		
		Note that due to the diamond condition, any flag $\Phi$ has a unique $j$-adjacent flag (denoted
		$\Phi^j$) for each $j = 0, 1, \ldots, n-1$.
		
		If $F$ is a $j$-face and $G$ is a $k$-face of a polytope with $F \leq G$, then the section $G/F$ is a
		($k-j-1$)-polytope itself. We can identify a face $F$ with the section $F/F_{-1}$; if $F$ is a $j$-face,
		then $F/F_{-1}$ is a $j$-polytope. We call the section $F_n/F$ the \emph{co-face at $F$}. The co-face
		at a vertex is also called a \emph{vertex-figure}. The section $F_{n-1}/F_0$ of a facet over a
		vertex is called a \emph{medial section}. Note that the medial section $F_{n-1}/F_0$ is both a facet
		of the vertex-figure $F_n/F_0$ as well as a vertex-figure of the facet $F_{n-1}/F_{-1}$.

		We sometimes need to work with \emph{pre-polytopes}, which are ranked partially ordered sets that
		satisfy the first, second, and fourth property above, but not necessarily the third. In this paper, all 
		of the pre-polytopes we encounter will be \emph{flag-connected}, meaning that if $\Phi$ and $\Psi$ are two 
		flags, there is a sequence of flags $\Phi = \Phi_0, \Phi_1, \ldots, \Phi_k = \Psi$ such
		that for $i = 0, \ldots, k-1$, the flags $\Phi_i$ and $\Phi_{i+1}$ are adjacent (but we do not require each
		flag to contain $\Phi \cap \Psi$). When working with pre-polytopes, we apply all the same terminology as 
		with polytopes. 
	
	\subsection{Regularity}
	
		For polytopes $\calP$ and $\calQ$, an \emph{isomorphism} from $\calP$ to $\calQ$ is an incidence- and rank-preserving bijection on the set of 
		faces. An isomorphism from $\calP$ to itself is an \emph{automorphism} of $\calP$. We denote the group of 
		all automorphisms of $\calP$ by $\G(\calP)$. There is a natural action of $\G(\calP)$ on the
		flags of $\calP$, and we say that $\calP$ is \emph{regular} if this action is transitive.
		For convex polytopes, this definition is equivalent to any of the usual definitions of regularity.
		
		Given a regular polytope $\calP$, fix a \emph{base flag} $\Phi$. Then the automorphism
		group $\G(\calP)$ is generated by the \emph{abstract reflections} $\rho_0, \ldots, \rho_{n-1}$,
		where $\rho_i$ maps $\Phi$ to the unique flag $\Phi^i$ that is $i$-adjacent to $\Phi$. 
		These generators satisfy $\rho_i^2 = \eps$ for all $i$, and $(\rho_i \rho_j)^2 = \eps$ for all $i$ 
		and $j$ such that $|i - j| \geq 2$.
		We say that $\calP$ has (\emph{Schl\"afli}) \emph{type} $\{p_1,\ldots,p_{n-1}\}$ if
		for each $i = 1, \ldots, n-1$ the order of $\rho_{i-1} \rho_i$ is $p_i$ (with $2 \leq p_i \leq \infty$).
		We also use $\{p_1, \ldots, p_{n-1}\}$ to represent the universal regular polytope of this type,
		which has an automorphism group with no relations other than those mentioned above.
		We denote the (Coxeter) group $\G(\{p_1, \ldots, p_{n-1}\})$ by $[p_1, \ldots, p_{n-1}]$.
		Whenever this universal polytope corresponds to a regular convex polytope, then the name used
		here is the same as the usual Schl\"afli symbol for that polytope (see \cite{coxeter}).
		
		For $I \subseteq \{0, 1, \ldots, n-1\}$ and a group $\G = \langle \rho_0, \ldots, \rho_{n-1} \rangle$,
		we define $\G_I := \langle \rho_i \mid i \in I \rangle$. 
		The strong flag-connectivity of polytopes induces the following {\em intersection property\/} in the group:
		\begin{equation}
		\label{eq:reg-int}
		\G_I \cap \G_J = \G_{I \cap J} 
		\;\; \textrm{ for } I,J \subseteq \{0, \ldots, n-1\}.
		\end{equation}

		In general, if $\G = \langle \rho_0, \ldots, \rho_{n-1} \rangle$ is a group such that each
		$\rho_i$ has order $2$ and such that $(\rho_i \rho_j)^2 = \eps$ whenever $|i - j| \geq 2$, then
		we say that $\G$ is a \emph{string group generated by involutions} (or \emph{sggi}). If
		$\G$ also satisfies the intersection property given above, then we call $\G$ a \emph{string
		C-group}. There is a natural way of building a regular polytope $\calP(\G)$ from a string
		C-group $\G$ such that $\G(\calP(\G)) = \G$ (see \cite[Ch. 2E]{arp}). Therefore, there is a one-to-one 
		correspondence between regular $n$-polytopes and string C-groups on $n$ specified generators.

	\subsection{Direct Regularity and Chirality}

		If $\calP$ is a regular polytope with automorphism group $\G(\calP)$ generated
		by $\rho_0, \ldots, \rho_{n-1}$, then the \emph{abstract rotations}
		\[ \s_i := \rho_{i-1} \rho_i \; (i = 1, \ldots, n-1) \]
		generate the \emph{rotation subgroup} $\GP$ of $\G(\calP)$, which has index at 
		most~$2$. We say that $\calP$ is \emph{directly regular} if this index is $2$. This is essentially an
		orientability condition; for example, the directly regular polyhedra correspond to orientable regular maps.
		The convex regular polytopes are all directly regular.
	
		We say that an $n$-polytope $\calP$ is \emph{chiral} if the action of $\G(\calP)$ on the flags
		of $\calP$ has two orbits such that adjacent flags are always in distinct orbits.
		For convenience, we define $\GP := \G(\calP)$ whenever $\calP$ is chiral.
		Given a chiral polytope $\calP$, fix a base flag $\Phi=\{F_{-1}, F_0, \ldots, F_n\}$.
		Then the automorphism group $\GP$ is generated by elements $\sigma_1, \ldots, \sigma_{n-1}$,
		where $\s_i$ acts on $\Phi$ the same way that $\rho_{i-1} \rho_i$ acts on the base flag of a regular
		polytope. That is, $\s_i$ sends $\Phi$ to $\Phi^{i,i-1}$. For $i < j$, we get that $(\s_i \cdots
		\s_j)^2 = \eps$. In analogy to regular polytopes, if the order of each $\s_i$ is $p_i$, 
		we say that the $\emph{type}$ of $\calP$ is $\{p_1, \ldots, p_{n-1}\}$.

		The automorphism groups of chiral polytopes and the rotation groups of directly regular polytopes satisfy 
		an intersection property analogous to that for string C-groups. Let $\G^+ := \GP = \langle \s_1, \ldots, 
		\s_{n-1} \rangle$ be the rotation group of a chiral or directly regular polytope $\calP$. For 
		$1 \leq i < j \leq n-1$ define $\tau_{i,j}:= \s_i \cdots \s_j$.
		By convention, we also define $\tau_{i,i} = \s_i$, and for $0 \leq i \leq n$, we define $\tau_{0,i} =
		\tau_{i,n} = \eps$.
		For $I \subseteq \{0, \ldots, n-1\}$, set
		\[ \G^+_I := \langle \tau_{i,j} \mid i \leq j \tx{ and } i-1, j \in I \rangle.  \]
		Then the \emph{intersection property} for $\G^+$ is given by:
		\begin{equation}
		\label{eq:chiral-int}
		\G^+_I \cap \G^+_J = \G^+_{I\cap J}  
		\;\; \textrm{ for } I,J \subseteq \{0, \ldots, n-1\}.
		\end{equation}
	
		If $\G^+$ is a group generated by elements $\s_1, \ldots, \s_{n-1}$ such that $(\s_i \cdots \s_j)^2 =
		\eps$ for $i < j$, and if $\G^+$ satisfies the intersection property above,
		then $\G^+$ is either the automorphism group of a chiral $n$-polytope or the rotation subgroup 
		of a directly regular polytope. In particular, it is the rotation subgroup
		of a directly regular polytope if and only if there is a group automorphism of $\G^+$ that sends $\s_1$ to 
		$\s_1^{-1}$, $\s_2$ to $\s_1^2 \s_2$, and fixes every other generator.
		
		Suppose $\calP$ is a chiral polytope with base flag $\Phi$ and with $\GP = \langle \s_1, \ldots,
		\s_{n-1} \rangle$. Let $\ch{\calP}$ be the chiral polytope with the same underlying face-set
		as $\calP$, but with base flag $\Phi^0$. Then $\G^+(\ch{\calP}) = \langle \s_1^{-1}, \s_1^2 \s_2, \s_3,
		\ldots, \s_{n-1} \rangle$. We call $\ch{\calP}$ the \emph{enantiomorphic form} or \emph{mirror image}
		of $\calP$. Though $\calP \simeq \ch{\calP}$, there is no automorphism of $\calP$ that takes
		$\Phi$ to $\Phi^0$.
		
		Let $\G^+ = \langle \s_1, \ldots, \s_{n-1} \rangle$, and let $w$ be a word in the free group on these 
		generators. We define the \emph{enantiomorphic} (or \emph{mirror image}) word $\ch{w}$ of $w$ to be the word obtained from $w$ by replacing every occurrence of $\s_1$ by $\s_1^{-1}$ and $\s_2$ by $\s_1^2\s_2$, 
		while keeping all $\s_j$ with $j\geq 3$ unchanged. Then if $\G^+$ is the rotation subgroup of a directly
		regular polytope, the elements of $\G^+$ corresponding to $w$ and $\ch{w}$ are conjugate in the
		full group $\G$.
		On the other hand, if $\G^+$ is the automorphism group of a chiral polytope, then $w$ and $\ch{w}$ need 
		not even have the same period. Note that $\ch{\ch{w}} = w$ for all words $w$.
		
		The sections of a regular polytope are again regular, and the sections of a chiral polytope are either
		directly regular or chiral. Furthermore, for a chiral $n$-polytope, all the $(n-2)$-faces and all the 
		co-faces at edges must be directly regular. As a consequence, if $\calP$ is a chiral polytope, it may be
		possible to extend it to a chiral polytope having facets isomorphic to $\calP$, but it will then
		be impossible to extend that polytope once more to a chiral polytope. 

		Chiral polytopes only exist in ranks 3 and higher.
		The simplest examples are the torus maps $\{4,4\}_{(b,c)}$, $\{3,6\}_{(b,c)}$ and 
		$\{6,3\}_{(b,c)}$, with $b,c\neq 0$ and $b\neq c$ (see \cite{cm}). These give rise to chiral $4$-polytopes
		having toroidal maps as facets and/or vertex-figures. More examples of chiral 4- and 5-polytopes can
		be found in \cite{chp}.
	
		Let $\calP$ and $\calQ$ be two polytopes (or flag-connected pre-polytopes) of the same rank, not 
		necessarily regular or chiral. A function $\gamma: \calP \to \calQ$ is called a
		\emph{covering} if it preserves incidence of faces, ranks of faces, and adjacency of flags; then $\gamma$ is
		necessarily surjective, by the flag-connectedness of $\calQ$. We say that $\calP$ \emph{covers} $\calQ$
		if there exists a covering $\gamma: \calP \to \calQ$.
		
		If a regular or chiral $n$-polytope $\calP$ has facets $\calK$ and vertex-figures $\calL$, we say 
		that $\calP$ is of \emph{type} $\{\calK,\calL\}$. If $\calP$ is of type $\{\calK, \calL\}$ and it covers
		every other polytope of the same type, then we say that $\calP$ is the \emph{universal polytope of
		type $\{\calK, \calL\}$}, and we simply denote it by $\{\calK, \calL\}$.

		If $\calP$ and $\calQ$ are chiral or directly regular $n$-polytopes, their rotation groups
		are both quotients of 
		\[ W^+ := [\infty, \ldots, \infty]^+ = \langle \s_1, \ldots, \s_{n-1} \mid (\s_i \cdots \s_j)^2
		= \eps \tx{ for $1 \leq i < j \leq n-1$} \rangle. \]
		Therefore there are normal subgroups $M$ and $K$ of $W^+$ such that $\GP = W^+/M$ and $\GQ = W^+/K$. Then
		$\calP$ covers $\calQ$ if and only if $M \leq K$.
	 
		Let $\calP$ be a chiral or directly regular polytope with $\GP = W^+/M$. We define 
		\[ \ch{M} = \{\ch{w} \mid w \in M\}. \]
		If $\ch{M} = M$, then $\calP$ is directly regular. Otherwise, $\calP$ is chiral, and $\GcP = W^+/\ch{M}$.
		
	\subsection{Duality}

		For any polytope $\calP$, we obtain the \emph{dual of $\calP$} (denoted $\calPd$) by simply reversing the 
		partial order. A \emph{duality} from $\calP$ to $\calQ$ is an anti-isomorphism; that is, a bijection 
		$\delta$ between the face sets such that $F < G$ in $\calP$ if and only if $\delta(F) > \delta(G)$ in $\calQ$.
		If a polytope is isomorphic to its dual, then it is called \emph{self-dual}.
		
		If $\calP$ is of type $\{\calK, \calL\}$, then $\calPd$ is of type $\{\calL^{\delta}, \calK^{\delta}\}$.
		Therefore, in order for $\calP$ to be self-dual, it is necessary (but not sufficient) that $\calK$
		is isomorphic to $\calL^{\delta}$ (in which case it is also true that $\calK^{\delta}$ is isomorphic
		to $\calL$).
		
		A self-dual regular polytope always possesses a duality that fixes the base flag. For chiral polytopes,
		this may not be the case. If a self-dual chiral polytope $\calP$ possesses a duality that sends the base flag to another flag in the same orbit (but reversing its direction), then there is a duality that fixes 
		the base flag, and we say that $\calP$ is \emph{properly self-dual} \cite{chiral-dual}. In this case, the 
		groups $\GP$ and $\GPd$ have identical presentations. If a self-dual chiral polytope has no duality that 
		fixes the base flag, then every duality sends the base flag to a flag in the other orbit, and $\calP$
		is said to be \emph{improperly self-dual}. In this case, the groups $\GcP$ and $\GPd$ have identical
		presentations instead. 
		
		If $\calP$ is a regular polytope with $\G(\calP) = \langle \rho_0, \ldots, \rho_{n-1} \rangle$, then the 
		group of $\calPd$ is $\G(\calPd) = \langle \rho_0', \ldots, \rho_{n-1}' \rangle$, where $\rho_i' = 
		\rho_{n-1-i}$. If $\calP$ is a directly regular or chiral polytope with $\GP = \langle \s_1, \ldots, 
		\s_{n-1} \rangle$, then the rotation group of $\calPd$ is $\G^+(\calPd) = \langle \s_1', \ldots,
		\s_{n-1}' \rangle$, where $\s_i' = \s_{n-i}^{-1}$. Equivalently, if $\GP$ has presentation
		\[ \langle \s_1, \ldots, \s_{n-1} \mid w_1, \ldots, w_k \rangle \]
		then $\G^+(\calPd)$ has presentation
		\[ \langle \s_1', \ldots, \s_{n-1}' \mid \delta(w_1), \ldots, \delta(w_k) \rangle, \]
		where if $w = \s_{i_1} \cdots \s_{i_j}$, then $\delta(w) = (\s_{n-i_1}')^{-1} \cdots (\s_{n-i_j}')^{-1}$.
		
		Suppose $\calP$ is a chiral or directly regular polytope with $\GP = W^+/M$. Then $\G^+(\calPd) =
		W^+/\delta(M)$, where $\delta(M) = \{ \delta(w) \mid w \in M \}$. If $\delta(M) = M$, then
		$\GP = \G^+(\calPd)$, so $\calP$ is properly self-dual.
		
		If $\calP$ is a chiral polytope, then $\ch{\calP^{\delta}}$ is naturally isomorphic
		to $\ch{\calP}^{\delta}$. Indeed, if $w$ is a word in the generators $\s_1, \ldots, \s_{n-1}$ of
		$\GP$, then 
		\[ \delta(\ch{w}) = (\s_1 \s_2 \cdots \s_{n-1}) \ch{\delta(w)} (\s_1 \s_2 \cdots \s_{n-1})^{-1}, \]
		so we see that the presentation for $\ch{\calP^{\delta}}$ is equivalent to that of $\ch{\calP}^{\delta}$.
		In particular, if $\GP = W^+/M$, then $\delta(\ch{M}) = \ch{\delta(M)}$ (since $M$ is a normal subgroup
		of $W^+$), and thus $\ch{\delta(\ch{\delta(M)})} = M$.

\section{Mixing polytopes}

	In this section, we will define the mix of two finitely presented groups, which naturally
	gives rise to a way to mix polytopes. The mixing operation is analogous to the join of hypermaps 
	\cite{ant2} and the parallel product of maps \cite{wilson}.
	
	Let $\G = \langle x_1, \ldots, x_n \rangle$ and $\G' =
	\langle x_1', \ldots, x_n' \rangle$ be groups with $n$ specified generators. Then the elements
	$z_i = (x_i, x_i') \in \G \times \G'$ (for $i = 1, \ldots, n$) generate a subgroup of
	$\G \times \G'$ that we call the \emph{mix} of $\G$ and $\G'$ and denote $\G \mix \G'$
	(see \cite[Ch.7A]{arp}).

	If $\calP$ and $\calQ$ are chiral or directly regular $n$-polytopes, we can mix their rotation groups.
	Let $\GP = \langle \s_1, \ldots, \s_{n-1} \rangle$ and $\GQ = \langle
	\s_1', \ldots, \s_{n-1}' \rangle$. Let $\beta_i = (\s_i, \s_i')$ for $i = 1, \ldots, n-1$.
	Then $\GP \mix \GQ = \langle \beta_1, \ldots, \beta_{n-1} \rangle$. We note that for $i < j$, we have
	$(\beta_i \cdots \beta_j)^2 = \eps$, so that the group $\GP \mix \GQ$ can be written as a quotient of $W^+$.
	In general, however, it will not have the intersection property (\eref{chiral-int}) with respect to its 
	generators $\beta_1, \ldots, \beta_{n-1}$. Nevertheless, it is possible to build a directly regular or
	chiral flag-connected pre-polytope from $\GP \mix \GQ$ using the method outlined in \cite{chiral}, and we 
	denote that pre-polytope $\calP \mix \calQ$
	and call it the \emph{mix} of $\calP$ and $\calQ$. Thus $\GPQ = \GP \mix \GQ$. If $\GP \mix \GQ$ satisfies
	the intersection property, then $\calP \mix \calQ$ is in fact a polytope.
	
	The following proposition is proved in \cite{const}:

	\begin{proposition}
	\label{prop:mix}
	Let $\calP$ and $\calQ$ be chiral or directly regular polytopes with $\GP = W^+/M$ and
	$\GQ = W^+/K$. Then $\GPQ \simeq W^+/(M \cap K)$.
	\end{proposition}
	
	Determining the size of $\GP \mix \GQ$ is often difficult for a computer unless $\GP$ and $\GQ$ are
	both fairly small. However, there is usually an easy way to indirectly calculate the size of the mix
	using the \emph{comix} of two groups. If $\G$ has presentation $\gp{x_1, \ldots, x_n}{R}$ and $\G'$ has 
	presentation $\gp{x_1', \ldots, x_n'}{S}$,
	then we define the comix of $\G$ and $\G'$, denoted $\G \comix \G'$, to be the group with presentation
	\[ \gp{x_1, x_1', \ldots, x_n, x_n'}{R, S, x_1^{-1}x_1', \ldots, x_n^{-1}x_n'}.\]
	Informally speaking, we can just add the relations from $\G'$ to $\G$, rewriting them to use
	$x_i$ in place of $x_i'$. 

	Just as the mix of two rotation groups has a simple description in terms of quotients of $W^+$, so
	does the comix of two rotation groups:

	\begin{proposition}
	\label{prop:comix}
	Let $\calP$ and $\calQ$ be chiral or directly regular polytopes with $\GP = W^+/M$ and
	$\GQ = W^+/K$. Then $\GcPQ \simeq W^+/MK$.
	\end{proposition}

	\begin{proof}
	Let $\GP = \langle \s_1, \ldots, \s_{n-1} \mid R \rangle$, and let $\GQ = \langle \s_1, \ldots, \s_{n-1} \mid
	S \rangle$, where $R$ and $S$ are sets of relators in $W^+$. 
	Then $M$ is the normal closure of $R$ in $W^+$ and $K$ is the normal closure of $S$ in $W^+$.
	We can write $\GcPQ = \langle \s_1, \ldots, \s_{n-1} \mid R \cup S \rangle$, so we want to show that $MK$ is
	the normal closure of $R \cup S$ in $W^+$. It is clear that $MK$ contains $R \cup S$, and since
	$M$ and $K$ are normal, $MK$ is normal, and so it contains the normal closure of $R \cup S$.
	To show that $MK$ is contained in the normal closure of $R \cup S$, it suffices to show that if
	$N$ is a normal subgroup of $W^+$ that contains $R \cup S$, then it must also contain $MK$. Clearly,
	such an $N$ must contain the normal closure $M$ of $R$ and the normal closure $K$ of $S$. Therefore,
	$N$ contains $MK$, as desired.
	\end{proof}

	Now we can determine how the size of $\GP \mix \GQ$ is related to the size of $\GP \comix \GQ$.
	
	\begin{proposition}
	\label{prop:mix-size}
	Let $\calP$ and $\calQ$ be finite chiral or directly regular $n$-polytopes. Then 
	\[ |\GP \mix \GQ| \cdot |\GcPQ| = |\GP| \cdot |\GQ|. \]
	\end{proposition}

	\begin{proof}
	Let $\GP = W^+/M$ and $\GQ = W^+/K$. Then by \pref{mix}, $\GP \mix \GQ = W^+/(M \cap K)$, and 
	by \pref{comix}, $\GcPQ = W^+/MK$.
	Let $\pi_1: \GP \mix \GQ \to \GP$ and $\pi_2: \GQ \to \GcPQ$ be the natural epimorphisms. Then
	$\ker \pi_1 \simeq M/(M \cap K)$ and $\ker \pi_2 \simeq MK/K \simeq M/(M \cap K)$. Therefore,
	we have that 
	\begin{align*}
	|\GP \mix \GQ| &= |\GP||\ker \pi_1| \\
	&= |\GP||\ker \pi_2| \\
	&= |\GP||\GQ|/|\GcPQ|,
	\end{align*}
	and the result follows.
	\end{proof}

	\begin{corollary}
	\label{cor:trivial-comix}
	Let $\calP$ and $\calQ$ be finite chiral or directly regular $n$-polytopes such that the group $\GP \comix \GQ$
	is trivial. Then $\GP \mix \GQ = \GP \times \GQ$.
	\end{corollary}
	
	The reason that \pref{mix-size} is so useful in calculating the size of $\GP \mix \GQ$ is that
	it is typically very easy for a computer to find the size of $\GP \comix \GQ$. Indeed, in many of the
	cases that come up in practice, it is easy to calculate $|\GP \comix \GQ|$ by hand just by combining
	the relations from $\GP$ and $\GQ$ and rewriting the presentation a little. 
	
	\subsection{Polytopality of the Mix}

		The mix of $\calP$ and $\calQ$ is polytopal if and only if $\GP \mix \GQ$ satisfies the intersection
		condition (\eref{chiral-int}). There is no general method for determining whether this condition
		is met. We start with the following result from \cite{const}.
			
		\begin{proposition}
		\label{prop:rel-prime-type}
		Let $\calP$ be a chiral or directly regular $n$-polytope of type $\{p_1, \ldots, p_{n-1}\}$, and let
		$\calQ$ be a chiral or directly regular $n$-polytope of type $\{q_1, \ldots, q_{n-1}\}$.
		If $p_i$ and $q_i$ are relatively prime for each $i = 1, \ldots, n-1$, then $\calP \mix \calQ$ is a 
		chiral or directly regular $n$-polytope of type $\{p_1 q_1, \ldots, p_{n-1} q_{n-1}\}$, and $\GPQ = 
		\GP \times \GQ$.
		\end{proposition}
		
		In general, when we mix $\calP$ and $\calQ$, we have to verify the full intersection property. But as we 
		shall see, some parts of the intersection property are automatic. Recall that for a subset $I$ of 
		$\{0, \ldots, n-1\}$ and a 
		rotation group $\G^+ = \langle \s_1, \ldots, \s_{n-1} \rangle$, we define 
		\[ \G^+_I = \langle \tau_{i,j} \mid i \leq j \tx{ and } i-1,j \in I \rangle, \]
		where $\tau_{i,j} = \s_i \cdots \s_j$.

		\begin{proposition}
		\label{prop:disjoint-int-prop}
		Let $\calP$ and $\calQ$ be chiral or directly regular $n$-polytopes, and let $I, J \subseteq \{0, \ldots, n-1\}$. 
		Let $\Lambda = \GP$, $\Delta = \GQ$, and $\G^+ = \Lambda \mix \Delta$.
		Then $\G^+_I \cap \G^+_J \leq \Lambda_{I \cap J} \times \Delta_{I \cap J}$.
		Furthermore, if $\G^+_I = \Lambda_I \times \Delta_I$ and $\G^+_J = \Lambda_J \times \Delta_J$,
		then $\G^+_I \cap \G^+_J = \Lambda_{I \cap J} \times \Delta_{I \cap J}$.
		\end{proposition}
		
		\begin{proof}
		Since $\G^+_I \leq \Lambda_I \times \Delta_I$ and $\G^+_J \leq \Lambda_J \times \Delta_J$, we have
		\begin{align*}
		\G^+_I \cap \G^+_J & \leq (\Lambda_I \times \Delta_I) \cap (\Lambda_J \times \Delta_J) \\
		& = (\Lambda_I \cap \Lambda_J) \times (\Delta_I \cap \Delta_J) \\
		& = \Lambda_{I \cap J} \times \Delta_{I \cap J},
		\end{align*}
		where the last line follows from the polytopality of $\calP$ and $\calQ$. This proves the first part.
		For the second part, we note that if $\G^+_I = \Lambda_I \times \Delta_I$ and $\G^+_J = \Lambda_J \times 
		\Delta_J$, then we get equality in the first line.
		\end{proof}
		
		\begin{corollary}
		\label{cor:polyhedra}
		Let $\calP$ and $\calQ$ be chiral or directly regular polyhedra. Then $\calP \mix \calQ$ is a 
		chiral or directly regular polyhedron.
		\end{corollary}

		\begin{proof}
		In order for $\calP \mix \calQ$ to be a polyhedron (and not just a pre-polyhedron), it must satisfy the 
		intersection property. For polyhedra, the only requirement is that $\langle \beta_1 \rangle \cap 
		\langle \beta_2 \rangle = \langle \eps \rangle$, which holds by \pref{disjoint-int-prop} by taking
		$I = \{0, 1\}$ and $J = \{1, 2\}$.
		\end{proof}
		
		\cref{polyhedra} is extremely useful. In addition to telling us that the mix of any two
		polyhedra is a polyhedron, it makes it simpler to verify the polytopality of the mix of $4$-polytopes,
		since the facets and vertex-figures of the mix are guaranteed to be polytopal.
		
\section{Mixing and Duality}

	We now come to the construction of properly and improperly self-dual polytopes. Let $\calP$ be
	a chiral or directly regular polytope, with $\GP = W^+/M$. Its dual $\calPd$ has rotation group
	$\GPd = W^+/\delta(M)$. By \pref{mix}, the rotation group of $\calP \mix \calPd$ is 
	$W^+/(M \cap \delta(M))$. Then since
	\[ \delta(M \cap \delta(M)) = \delta(M) \cap \delta(\delta(M)) = \delta(M) \cap M, \]
	we see that $\calP \mix \calPd$ is properly self-dual.
	
	Similarly, suppose that $\calP$ is a chiral polytope with $\GP = W^+/M$. Then $\ch{\calPd}$, the mirror
	image of its dual, has rotation group $\G^+(\ch{\calPd}) = W^+/\ch{\delta(M)}$. Let $\calQ = \calP \mix
	\ch{\calPd}$. Then $\GQ = W^+/(M \cap \ch{\delta(M)})$. We see that
	\[ \ch{\delta(M \cap \ch{\delta(M)})} = \ch{\delta(M)} \cap \ch{\delta(\ch{\delta(M)})} = \ch{\delta(M)}
	\cap M, \]
	so $\calQ = \ch{\calQ^{\delta}}$. If $\calQ$ is directly regular, then it is (properly) self-dual.
	Otherwise, if $\calQ$ is chiral, then it is improperly self-dual.

	Under what conditions is $\calP \mix \calPd$ or $\calP \mix \ch{\calPd}$ polytopal? If $\calP$ is a polyhedron,
	then $\calP \mix \calPd$ and $\calP \mix \ch{\calPd}$ are always polytopal by \cref{polyhedra}. For 
	polytopes in ranks 4 and higher, we can try to apply the results of the previous section. For example,
	by specializing \pref{rel-prime-type}, we get the following result.
	
	\begin{proposition}
	\label{prop:rel-prime-type-dual}
	Let $\calP$ be a chiral or directly regular $n$-polytope of type $\{p_1, \ldots, p_{n-1}\}$ such that for all
	$i=1, \ldots, n-1$ we have $\gcd(p_i, p_{n-i}) = 1$. Then $\calP \mix \calP^{\delta}$ is a properly self-dual chiral or directly regular $n$-polytope of type $\{p_1 q_1, \ldots, p_{n-1} q_{n-1}\}$, and $\G^+(\calP \mix
	\calPd) = \GP \times \GPd$.
	\end{proposition}
	
	This result is nice because it requires very little information about $\calP$. However,
	it is fairly restrictive. In particular, if $n$ is even, then $p_{n/2} = p_{n - n/2}$, and so the condition
	on the numbers $p_i$ is never satisfied. In this case, having certain numbers $p_i$ relatively prime
	to $p_{n-i}$ is actually an impediment to polytopality.
	
	\begin{theorem}
	Let $\calP$ be a chiral or directly regular $n$-polytope of type $\{p_1, \ldots, p_{n-1}\}$, and suppose
	that $n$ is even. Let $m = n/2$, and suppose that $p_{m-1}$ and $p_{m+1}$ are relatively prime, and
	that $p_m \geq 3$. Then $\calP \mix \calPd$ is not a polytope.
	\end{theorem}
	
	\begin{proof}
	Let $\GP = \langle \s_1, \ldots, \s_{n-1} \rangle$, $\GPd = \langle \s_1', \ldots, \s_{n-1}' \rangle$,
	and $\beta_i = (\s_i, \s_i')$ for each $i \in \{1, \ldots, n-1\}$. To show that $\calP \mix \calPd$
	is not polytopal, it suffices to show that
	\[ \langle \beta_{m-1}, \beta_m \rangle \cap \langle \beta_m, \beta_{m+1} \rangle \neq \langle \beta_m \rangle. \]
	Now, since $p_{m-1}$ and $p_{m+1}$ are relatively prime, there is an integer $k$ such that $kp_{m-1} \equiv
	1$ (mod $p_{m+1}$). Then since the order of $\s_{m-1}$ is $p_{m-1}$ and the order of $\s_{m-1}'$ is
	$p_{m+1}$, we see that
	\[ \beta_{m-1}^{kp_{m-1}} = (\s_{m-1}^{kp_{m-1}}, (\s_{m-1}')^{kp_{m-1}}) = (\eps, \s_{m-1}'), \]
	and therefore
	\[ (\beta_{m-1}^{kp_{m-1}} \beta_m)^2 = (\s_m^2, (\s_{m-1}' \s_m')^2) = (\s_m^2, \eps), \]
	since we have $(\s_i' \s_{i+1}')^2 = \eps$ for any $i \in \{1, \ldots, n-2\}$. Thus,
	$(\s_m^2, \eps) \in \langle \beta_{m-1}, \beta_m \rangle$.
	Similarly, there is an integer $k'$ such that $k'p_{m+1} \equiv 1$ (mod $p_{m-1}$), and thus
	\[ (\beta_m \beta_{m+1}^{k' p_{m+1}})^2 = (\s_m^2, (\s_m' \s_{m+1}')^2) = (\s_m^2, \eps). \]
	Therefore, $(\s_m^2, \eps) \in \langle \beta_m, \beta_{m+1} \rangle$ as well. So we see that
	\[ (\s_m^2, \eps) \in \langle \beta_{m-1}, \beta_m \rangle \cap \langle \beta_m, \beta_{m+1} \rangle.\]
	On the other hand, since the elements $\s_m$ and $\s_m'$ both have order $p_m$, which is at least $3$,
	we clearly have that $(\s_m^2, \eps) \not \in \langle \beta_m \rangle$, and that proves the claim.
	\end{proof}
	
	For example, if $\calP$ is the locally toroidal chiral polytope $\{\{6, 3\}_{(b, c)}, \{3, 5\}\}$ or
	$\{\{4, 4\}_{(b,c)}, \{4, 3\}\}$ (with $bc(b-c) \neq 0$), then $\calP \mix \calPd$ is not polytopal. 
	
	There are cases where the mix of a chiral $4$-polytope with its dual is polytopal. Here is one
	applicable result.
	
	\begin{proposition}
	\label{prop:4-polytopality}
	Let $\calP$ be a finite chiral or directly regular $4$-polytope of type $\{p, q, r\}$, with facets $\calK$
	and vertex-figures $\calL$. If $q$ is prime and if $q^2$ does not divide $|\GK \mix \G^+(\calL^{\delta})|$,
	then $\calP \mix \calPd$ is polytopal.
	\end{proposition}
	
	\begin{proof}
	Let $\GP = \langle \s_1, \s_2, \s_3 \rangle$, $\GPd = \langle \s_1', \s_2', \s_3' \rangle$, and
	$\GP \mix \GPd = \langle \beta_1, \beta_2, \beta_3 \rangle$, where $\beta_i = (\s_i, \s_i')$. 
	Then $\calP \mix \calPd$ is polytopal if and only if $\langle \beta_1, \beta_2 \rangle \cap \langle \beta_2, 
	\beta_3 \rangle = \langle \beta_2 \rangle$. From \pref{disjoint-int-prop}, we know that
	\[ \langle \beta_1, \beta_2 \rangle \cap \langle \beta_2, \beta_3 \rangle \leq \langle \s_2 \rangle \times
	\langle \s_2' \rangle. \]
	Let $N = |\langle \beta_1, \beta_2 \rangle \cap \langle \beta_2, \beta_3 \rangle|$. Then $N$ must divide
	$|\langle \beta_1, \beta_2 \rangle|$, which is $|\GK \mix \G^+(\calL^{\delta})|$, and it also must
	divide $|\langle \s_2 \rangle \times \langle \s_2' \rangle|$, which is $q^2$. Since $q^2$ does not divide
	$|\GK \mix \G^+(\calL^{\delta})|$, we must have $N \neq q^2$. So $N$ must be a proper divisor of $q^2$.
	Since we clearly have $\beta_2 \in \langle \beta_1, \beta_2 \rangle \cap \langle \beta_2, \beta_3 \rangle$,
	we see that $N$ must be at least $q$. Therefore, since $q$ is prime, we must have $N = q$, in which case
	\[ \langle \beta_1, \beta_2 \rangle \cap \langle \beta_2, \beta_3 \rangle = \langle \beta_2 \rangle. \]
	Thus $\calP \mix \calPd$ is polytopal.
	\end{proof}

	We will see an example that uses this result in Section~6.
	
\section{Chirality of Self-Dual Mixes}

	We now set aside the question of whether $\calP \mix \calPd$ is polytopal and focus on determining
	conditions for which $\calP \mix \calPd$ is chiral. All of the results of this section can also
	be applied to $\calP \mix \ch{\calPd}$ with little or no modification.

	\begin{proposition}
	\label{prop:chiral-mix}
	Let $\calP$ be a chiral polytope and let $\calQ$ be a chiral or directly regular polytope. 
	If $\calP \mix \calQ$ is directly regular, then it covers $\calP \mix \ch{\calP}$.
	\end{proposition}
	
	\begin{proof}
	Let $\GP = W^+/M$ and let $\GQ = W^+/K$. Then $\G^+(\calP \mix \calQ) = W^+/(M \cap K)$ and 
	$\G^+(\calP \mix \ch{\calP}) = W^+/(M \cap \ch{M})$. If $\calP \mix \calQ$ is directly regular, then 
	$M \cap K = \ch{M \cap K} = \ch{M} \cap \ch{K}$. Therefore, $M \cap K \leq M \cap \ch{M}$,
	and thus $\calP \mix \calQ$ covers $\calP \mix \ch{\calP}$.
	\end{proof}
	
	\begin{proposition}
	\label{prop:self-dual-chirality}
	Let $\calP$ be a finite chiral polytope. If $|\GP \comix \GPd| > |\GP \comix \GcP|$,
	then $\calP \mix \calPd$ is chiral.
	\end{proposition}
	
	\begin{proof}
	If $|\GP \comix \GPd| > |\GP \comix \GcP|$, then
	$|\GP \mix \GPd)| < |\GP \mix \GcP|$, by \pref{mix-size}. In particular,
	$\calP \mix \calPd$ cannot cover $\calP \mix \ch{\calP}$, and so $\calP \mix \calPd$ is chiral by
	\pref{chiral-mix}.
	\end{proof}
	
	By taking into account the Schl\"afli symbol of $\calP$, we obtain a slightly stronger result.

	\begin{theorem}
	\label{thm:self-dual-chirality}
	Let $\calP$ be a finite chiral polytope of type $\{p_1, \ldots, p_{n-1}\}$. Define $\ell_i = 
	\lcm(p_i, p_{n-i})$ for $i = 1, \ldots, n-1$, and let $\ell = \lcm(\ell_1/p_1, \ldots, \ell_{n-1}/p_{n-1})$. 
	If 
	\[ |\GP \comix \GcP| < \ell \, |\GP \comix \GPd|, \] 
	then $\calP \mix \calPd$ is chiral.
	\end{theorem}
	
	\begin{proof}
	Suppose $\calP \mix \calPd$ is directly regular. Then $\GP \mix \GPd$ covers $\GP \mix 
	\GcP$, by \pref{chiral-mix}. Let $\pi$ be the corresponding natural epimorphism. 
	Now, $\calP \mix \calPd$ is of type $\{\ell_1, \ldots, 
	\ell_{n-1}\}$, while $\calP \mix \ch{\calP}$ is of type $\{p_1, \ldots, p_{n-1}\}$. Let $\G^+(\calP \mix \calPd)
	= \langle \s_1, \ldots, \s_{n-1} \rangle$. Then we have that $\s_i^{p_i} \in \ker \pi$ for each $i = 1, 
	\ldots, n-1$. So $\langle \s_1^{p_1}, \ldots, \s_{n-1}^{p_{n-1}} \rangle \leq \ker \pi$. Now, the order of 
	$\s_i^{p_i}$ in $\G^+(\calP \mix \calPd)$ is $\ell_i / p_i$ since the order of $\s_i$ is $\ell_i$ and $p_i$ 
	divides $\ell_i$. Then $\ker \pi$ contains elements of order $\ell_i / p_i$ for $i = 1, \ldots, n-1$, and thus 
	it has size at least $\ell = \lcm(\ell_1/p_1, \ldots, \ell_{n-1}/p_{n-1})$. Now, we have that $|\GP \mix \GPd|
	= |\ker \pi||\GP \mix \GcP|$, and therefore 
	\[ |\GP \comix \GcP| = |\ker \pi| |\GP \comix \GPd| \geq \ell \, |\GP \comix \GPd|, \]
	proving the desired result.
	\end{proof}

	Finally, we establish a result that relies on the fact that $\ch{\calP^{\delta}} = \ch{\calP}^{\delta}$.
	
	\begin{theorem}
	\label{thm:big-chirality}
	Let $\calP$ be a finite chiral polytope, and suppose that
	\[ \left( \frac{|\GP \mix \GcP|}{|\GP|} \right)^2 > \left| \left(\GP \mix \GcP 
	\vphantom{\G^+(\ch{\calP}^{\delta})} \right) \comix 
	\left(\GPd \mix \G^+(\ch{\calP}^{\delta})\right) \right|. \]
	Then $\calP \mix \calPd$ is chiral.
	\end{theorem}
	
	\begin{proof}
	Suppose that $\calP \mix \calPd$ is directly regular. Then $(\calP \mix \calPd) \mix (\ch{\calP} \mix \ch{\calPd}) = \calP \mix \calPd$. Now, we have that
	\begin{align*}
	|\GP \mix \GPd| &= |\GP \mix \GPd \mix \GcP \mix \G^+(\ch{\calP^{\delta}})| \\
	&= |(\GP \mix \GcP) \mix (\GPd \mix \G^+(\ch{\calP}^{\delta}))| \\
	&= \frac{|\GP \mix \GcP| |\GPd \mix \G^+(\ch{\calP}^{\delta})|}{|(\GP \mix \GcP) \comix (\GPd \mix \G^+(\ch{\calP}^{\delta}))|} \\
	&= \frac{|\GP \mix \GcP|^2}{|(\GP \mix \GcP) \comix (\GPd \mix \G^+(\ch{\calP}^{\delta}))|},
	\end{align*}
	where the third line follows from \pref{mix-size}.
	Rearranging, we get that
	\begin{align*}
	|(\GP \mix \GcP) \comix (\GPd \mix \G^+(\ch{\calP}^{\delta}))| &= \frac{|\GP \mix \GcP|^2}{|\GP \mix \GPd|} \\
	& \geq \frac{|\GP \mix \GcP|^2}{|\GP|^2},
	\end{align*}
	and the result follows.
	\end{proof}
	
	\begin{corollary}
	\label{cor:big-chirality}
	Let $\calP$ be a chiral polytope of type $\{p_1, \ldots, p_{n-1}\}$, and suppose that
	\[ \left( \frac{|\GP \mix \GcP|}{|\GP|} \right)^2 > |[p_1, \ldots, p_{n-1}]^+ \comix [p_{n-1}, \ldots, p_1]^+|. \]
	Then $\calP \mix \calPd$ is chiral.
	\end{corollary}

	\begin{proof}
	Since $\calP$ is of type $\{p_1, \ldots, p_{n-1}\}$, so are $\ch{\calP}$ and $\calP \mix \ch{\calP}$.
	Similarly, $\calPd \mix \ch{\calP}^{\delta}$ is of type $\{p_{n-1}, \ldots, p_1\}$. 
	Therefore, $(\GP \mix \GcP) \comix (\GPd \mix \G^+(\ch{\calP}^{\delta}))$ is a quotient of
	$[p_1, \ldots, p_{n-1}]^+ \comix [p_{n-1}, \ldots, p_1]^+$, and the result follows from \tref{big-chirality}.
	\end{proof}
	
	We now look at a few broad classes of examples where $\calP \mix \calPd$ is guaranteed to be a chiral,
	self-dual polytope.
	
	\begin{theorem}
	Let $\calP$ be a finite chiral polyhedron of type $\{p, q\}$. Let $\ell_1 = \lcm(p,q)$, and suppose that 
	$|\GP \comix \GcP| < \ell_1^2 / pq \, |\GP \comix \GPd|$. Then $\calP \mix \calPd$ is a properly self-dual chiral 
	polyhedron of type $\{\ell_1, \ell_1\}$.
	\end{theorem}
	
	\begin{proof}
	From Corollary~\ref{cor:polyhedra}, we know that $\calP \mix \calPd$ is a chiral or directly regular polyhedron. 
	Now, we apply \tref{self-dual-chirality}. We have that $\ell = \lcm(\ell_1/p, \ell_1/q) = \ell_1^2/pq$, and therefore, $\calP \mix \calPd$ is chiral.
	\end{proof}
	
	\begin{theorem}
	Let $\calP$ be a finite chiral polytope of odd rank of type $\{p_1, \ldots, p_{n-1}\}$. Suppose 
	$\gcd(p_i, p_{n-i}) = 1$ for $i = 1, \ldots, n-1$, and suppose that $|\GP \comix \GcP| < 
	\lcm(p_1, \ldots, p_{n-1})$. Then $\calP \mix \calPd$ is a properly self-dual chiral polytope of type 
	$\{p_1 p_{n-1}, p_2 p_{n-2}, \ldots, p_{n-1} p_1\}$, and with group $\GP \times \GPd$.
	\end{theorem}
	
	\begin{proof}
	With the given conditions, Proposition~\ref{prop:rel-prime-type-dual} applies to show us that $\calP \mix 
	\calPd$ is a polytope with group $\GP \times \GPd$. To prove chirality, we apply 
	\tref{self-dual-chirality}, noting that $\ell_i = p_i p_{n-i}$, $\ell = \lcm(p_1, \ldots, p_{n-1})$,
	and $|\GP \comix \GPd| = 1$.
	\end{proof}
	
\section{Self-dual Chiral Polytopes}

	Now we will apply the results of the preceding sections to build some concrete examples of self-dual
	chiral polytopes.
	
	If $\calP$ is a chiral polytope with simple automorphism group, then $\GP \comix \GcP$ is trivial
	\cite{const}. If $\calP$ is not already self-dual, then $\GP \comix \GPd$ is not
	trivial, and therefore $\calP \mix \calPd$ must be chiral by \pref{self-dual-chirality}. The question
	of polytopality of $\calP \mix \calPd$ must still be addressed, but if $\calP$ is a polyhedron, for example, then
	polytopality follows from \cref{polyhedra}. There are many examples of such polyhedra; for example,
	in \cite{aj}, the authors give several examples of chiral polyhedra whose automorphism group is
	the Mathieu group $M_{11}$.

	Next we consider the simplest chiral polyhedra: the torus maps. Since the torus map $\{4, 4\}_{(b, c)}$ is
	already (improperly) self-dual, we work only with $\{3, 6\}_{(b,c)}$ and its dual. Let $\calP = \{3, 6\}_{(b,c)}$,
	where $m := b^2 + bc + c^2$ is a prime and $m \geq 5$. (The primality of $m$ is not essential, but it makes
	some of our calculations easier.) We have that $|\GP| = 6m$ and $|\GP \mix \GcP| = 6m^2$ \cite{cox-index}.
	Now, the dual of $\calP$ is $\{6, 3\}_{(b,c)}$, so $\calP \comix \calPd$ is a quotient of $\{3, 3\}$. 
	This is already enough to conclude that $\calP \mix \calPd$ is chiral (using \cref{big-chirality}),
	but we also want to determine the full structure of $\calP \mix \calPd$, so we need to calculate the size
	of $\GP \comix \GPd$ directly. Since $m$ is prime, $b$ and $c$ must be coprime, and in particular, at least
	one of them must be odd. We can assume that $b$ is odd by changing from $\calP = \{3, 6\}_{(b,c)}$ to
	$\ch{\calP} = \{3, 6\}_{(c, b)}$ if necessary. Now, in $\GP \comix \GPd$, we have the relation
	\[ (\s_1 \s_2^{-1} \s_1^{-1} \s_2)^b (\s_2 \s_1 \s_2^{-1} \s_1^{-1})^c = \eps. \]
	Using the facts that $(\s_1 \s_2)^2 = \s_1^3 = \s_2^3 = \eps$ and that $b$ is odd, we can conclude that
	\[ \s_2 \s_1 (\s_2 \s_1^{-1} \s_2)^c = \eps. \]
	Conjugating both sides by $\s_2$ and making a few more easy reductions, we get that either
	\begin{align*} 
	\s_2^{-1} \s_1 \s_2 \s_1^{-1} &= \eps \tx{ if $c$ is odd,} \\
	\s_2^{-1} \s_1 \s_2^{-1} &= \eps \tx{ if $c$ is even.}
	\end{align*}
	In the first case, we see that $\s_1 \s_2 = \s_2 \s_1$, and since we also have $(\s_1 \s_2)^2 = \eps$, we
	see that $\s_1 = \s_2^{-1}$. In the second case, we also directly get that $\s_1 = \s_2^{-1}$, and therefore
	$\s_1 \s_2 = \s_2 \s_1$. In any case, the extra relation from $\{6, 3\}_{(b,c)}$ is rendered redundant,
	and we see that $\GP \comix \GPd$ has order $3$.
	
	We can now determine the full structure of $\calP \mix \calPd$. We have that $|\GP \mix \GPd| = |\GP|^2 / 3 =
	12m^2$, and therefore, $\calP \mix \calPd$ has $24m^2$ flags. Since $\calP \mix \calPd$ is of type $\{6, 6\}$,
	it must have $2m^2$ vertices, $6m^2$ edges, and $2m^2$ $2$-faces. 
	
	The previous analysis also works for $\calP 
	\mix \ch{\calPd}$, and we get an improperly self-dual chiral
	polytope with the same number of flags, vertices, etc. as $\calP \mix \calPd$.

	Finally, we present an example of a chiral $4$-polytope that we can self-dualize.
	Let $\calP$ be the polytope of type $\{\{6, 3\}_{(b, c)}, \{3, 3\}\}$ with group $L_2(m)$ $(= PSL(2, m))$, where
	$m = b^2 + bc + c^2$ is prime and $m \equiv 1$ (mod 12) \cite{SW2}. First, we want to show that $\calP \mix 
	\calPd$ is polytopal. By \pref{4-polytopality}, it suffices to show that $9$ does not divide
	$|[6, 3]^+_{(b, c)} \mix [3, 3]^+|$. The argument used above to show that $|[6, 3]^+_{(b, c)} \comix
	[3, 6]^+_{(b, c)}| = 3$ can be applied here to show that $|[6, 3]^+_{(b, c)} \comix [3, 3]^+| = 3$ as well.
	Then 
	\[ |[6, 3]^+_{(b, c)} \mix [3, 3]^+| = |[6, 3]^+_{(b, c)}| \cdot |[3, 3]^+| / 3 = 24m. \]
	Since $m$ is a prime and $m \neq 3$, $9$ does not divide $24m$, and thus $\calP \mix \calPd$ is polytopal.
	
	To show that $\calP \mix \calPd$ is chiral, it suffices to show that the facets $\{6, 3\}_{(b, c)} \mix
	\{3, 3\}$ are chiral. As mentioned above, $|[6, 3]^+_{(b, c)} \mix [6, 3]^+_{(c, b)}| = 6m^2$. If
	$\{6, 3\}_{(b, c)} \mix \{3, 3\}$ is directly regular, it must cover $|[6, 3]^+_{(b, c)} \mix [6, 3]^+_{(c, b)}$,
	by \pref{chiral-mix}. This can only happen if $6m^2$ divides $24m$, which does not happen for $m > 4$. Thus
	we see that the facets of $\calP \mix \calPd$ are chiral, and therefore, so is the whole polytope.
	
	Since $|[6, 3]^+_{(b, c)} \mix [3, 3]^+| = 24m$, we see that the facets of $\calP \mix \calPd$ are of
	type $\{6, 3\}$ with $48m$ flags, and thus the facets have $8m$ vertices, $12m$ edges, and $4m$ $2$-faces.
	Therefore, the facets have Euler characteristic $0$ and so they are torus maps; in fact, the
	facets are $\{6, 3\}_{(2b, 2c)}$. The vertex-figures of $\calP \mix \calPd$ are the dual of the facets,
	so they are equal to $\{3, 6\}_{(2b, 2c)}$. Thus, $\calP \mix \calPd$ is a properly self-dual chiral polytope of type $\{\{6, 3\}_{(2b, 2c)}, \{3, 6\}_{(2b, 2c)}\}$ with automorphism group $L_2(m) \times L_2(m)$.

\end{document}